\documentclass[11pt]{amsart}

\usepackage{amssymb, amsthm, amsmath}
\usepackage{bbm}
\usepackage{xcolor}
\usepackage{graphicx}
\usepackage{mathtools}

\usepackage{pgf,tikz}
\usepackage{mathrsfs}
\usetikzlibrary{arrows}
\usepackage{pgfplots}
\pgfplotsset{compat=1.15}

\newtheorem{theorem}{Theorem}[section]

\newtheorem{lemma}[theorem]{Lemma}

\newtheorem*{theorem*}{Theorem}{\bf}{\it}
\newtheorem*{proposition*}{Proposition}{\bf}{\it}
\newtheorem*{observation*}{Observation}{\bf}{\it}
\newtheorem*{lemma*}{Lemma}{\bf}{\it}
{\bf}{\it}

\theoremstyle{definition}

\theoremstyle{remark}
\newtheorem{remark}[theorem]{Remark}

\newcommand{\R}{\mathbb R}

\newcommand*{\QEDA}{\hfill\ensuremath{\square}}
\newcommand{\abs}[1]{\vert #1 \vert}
\def\N{ {\mathbb{N}} }
\newcommand{\ds}{\displaystyle}
\newcommand{\defeq}{\vcentcolon=}

\def\XXint#1#2#3{{\setbox0=\hbox{$#1{#2#3}{\int}$ }
\vcenter{\hbox{$#2#3$ }}\kern-.6\wd0}}

\begin{document}
\title[Elliptic Adaptation of the log log Thoerem]{An elliptic adaptation  of ideas of Carleman and Domar from complex analysis related to Levinson's log log theorem}

\author{A. Logunov}
\address{Alexander Logunov: Department of Mathematics, Princeton University, Princeton, NJ, USA}
\email{log239@yandex.ru}
\author{H.  Papazov}
\address{Hristo Papazov: Department of Mathematics, Princeton University, Princeton, NJ, USA}
\email{h.g.papazov@gmail.com}

\begin{abstract}
  Using  the three  balls inequality, we adapt the elegant ideas of Carleman and Domar from complex analysis to linear elliptic PDE and generalize the classical Levinson's loglog theorem.
\end{abstract}

\maketitle

\section{Introduction.}

 We start with a problem of improving estimates for solutions to  elliptic PDE.
 Assume that $u$ is a harmonic function in $\mathbb{R}^3$, and we know that $|u(x)|$ is bounded by a given majorant $M(|x_3|)$ outside the plane $\{x \in \mathbb{R}^3: x_3=0 \}$.
 The majorant is allowed to explode when $x_3$ tends to $0$. For instance, one can take $M(t)=e^{1/t}$, $t>0$.  Why does the harmonicity of $u$ across the plane imply a better bound $|u(0)|< 10^6$?
 If the majorant $M(t)$ was $1/\sqrt t$ in place of $e^{1/t}$, one could use the mean value property of $u$ in the unit ball with center at the origin to see immediately that
 $|u(0)|$ can be estimated by the average of $M(|x_3|)$ over the unit ball.
  The statement is true for $M(t)=e^{1/t}$ though the mean-value argument fails.
  Moreover, one can take as $M(t)$ any measurable function with $\int_{0}^{1}\log^+\log^+ M(t)<+\infty$, and then $|u(x)|$ will be bounded in $B_{1/2}(0)$ by some constant $C_M$, which depends on $M$ only. Here we use the notation $\log^+(x)= \begin{cases} \log x, & x \geq 1 \\ 0, & x \leq 1    \end{cases}.$
  
 The first result of this kind is due to Carleman, who proved the following generalization of Liouville's theorem for holomorhpic functions.
  \begin{theorem}[\cite{C}]
    Let $f$ be a holomorphic function in $\mathbb{C}$ and let $M(\theta)$ be a measurable function on the unit circle with
    $$ \int_{0}^{2\pi}\log^+\log^+ M(\theta)<+\infty.$$
    If on each ray passing through the origin $|f(re^{i\theta})|\leq M(\theta)$, then
    $f\equiv const$.
  \end{theorem}
  Carleman's argument \cite{C} is elegant and short.  However his work was lost for decades and some of the further works on this subject were significantly more technical and long.
 Levinson  proved a local version of this fact, see \cite{Le} (the proof required additional regularity assumptions on the majorant).  The term "Levinson's $\log\log$ type theorems" is used in complex analysis though the lost original paper \cite{C} of Carleman was earlier.
   Domar's elegant proof \cite{K88,D1,D2} of Levinson's local theorem is barehanded and is close in spirit to Carleman's approach. Instead of writing a differential inequality for the growth of a function, Domar establishes a discrete sequence of inequalities. It makes the argument more flexible. We learned Domar's argument from Koosis's textbook \cite{K88}.
    
 Beurling found an independent proof \cite{Be} of Levinson's theorem when he was working on a problem of analytic extension across the boundary.  Beurling also showed that the result is sharp in the following sense.
  \begin{theorem}
  Given a positive monotonically decreasing function $M(t):(0,1)\to \mathbb{R}^+$ with $M(t)\to + \infty$ as $t \to +0$ and $\int_0^1 \log^+\log^+ M = +\infty$,
   there is a sequence of holomorphic functions $f_n$ in a unit disc $B_1(0)\subset \mathbb{C}$ such that
   $\displaystyle |f_n(0)| \to \infty$ and $f_n(z)\leq M(|\Im z|)$ for $z\in B_1(0)$, $|\Im z|\neq 0$.
  \end{theorem}

  Rashkovskii (see \cite{R91,R09}) showed that the Levinson's loglog theorem and related results by Wolf \cite{W},  Sjoberg \cite{S}, and Matsaev \cite{M60, M76} for holomorphic functions can be proved in a uniform way with the help of estimates of harmonic measure.
  This approach as well as Carleman's argument is relying on the fact that $\log |f|$ is subharmonic for any holomorphic function $f$. We would like to note
  that this fact is not available for solutions of elliptic PDE.
  
  However, the main message of this article  is that the phenomenon described in Levinson's theorem is actually of elliptic nature.  Here is a more general question.
 Let $B$ be an open ball in $\mathbb{R}^n$ and let $P$ be a differential operator in $B$. Given a subset $S\subset B$ and a function $M: \Omega \to \mathbb{R}_+\cup \{+ \infty\}$ that tends to infinity when we approach $S$, define $F_M$ to be the family of solutions $u$ to $Pu=0$ in $B$ such that $u(x)\leq M(x)$ in $\Omega$. The question is whether the functions in $F_M$ admit a uniform bound in a twice smaller ball $\frac{1}{2}B$ by some constant $C=C(M,P)$?

 We give a positive answer in a form of Levinson's loglog theorem for some class of elliptic operators. An elliptic tool, which is used here, is the three balls inequality. Sometimes it replaces a very simple and powerful fact from complex analysis stating that $\log|f|$ is subharmonic for a holomorphic function  $f$.

Generalizing the statement about holomorphic functions to elliptic PDE is non-trivial even in the case for harmonic functions in the three dimensional Euclidean space. The question whether Levinson's loglog type theorem holds for harmonic functions appeared in connection to the study of universal Taylor series, see \cite{GK14}. In the three dimensional Euclidean space, one can use symmetries of the Laplace equation to reduce Levinson's loglog type question for harmonic functions to a two-dimensional problem. This reduction was done in \cite{L}. However, it is clear that in the case of equations with variable coefficients the symmetrization techniques do not work.  
The main ingredient, which can be combined with  the ideas of Carleman and Domar, is the three balls inequality.

\section*{Acknowledgements.}
 The Packard Foundation sponsored the summer research project at Princeton University for the second author in 2020. The first author was supported in part by the Packard Fellowship and Sloan Fellowship. The first author is extremely grateful to the School of Mathematical Sciences and Institute for Advanced Study at Tel Aviv University for organizing the new visiting program for scholars "IAS Outstanding Junior Fellows" despite epidemic, bureaucratic and travel restrictions, which helped to stay afloat in very uncertain times.

 \section{Three balls inequality.} \label{sec:3b}
 In this section, we formulate the three balls inequality and overview for which classes of elliptic operators it holds.
Given a ball $B$ in $\mathbb{R}^n$ and a positive constant  $k$, we denote by $kB$ the ball with same center as $B$ and $k$ times bigger radius.
 
 We say that three balls inequality holds for solutions $u$ of  $Pu=0$ in a bounded domain $\Omega \subset \mathbb{R}^n$ if

\begin{equation} \label{eq:3b}
\sup_B |u| \leq C (\sup_{\frac{1}{2}B} |u|)^{\alpha} (\sup_{2B} |u|)^{1-\alpha}
\end{equation} 
 for any ball $B$ with $2B\subset \Omega$ and any solution 
 $u$ to $Pu=0$ in $\Omega$, where  $C>0$ and $\alpha \in (0,1)$ depend on $P$ and $\Omega$ but not on $u$ and $B$.
 
\begin{remark} We note that in the three balls inequality one can change the ratios between radii:
\begin{equation} \label{eq:3b*}
\sup_B |u| \leq C (\sup_{aB} |u|)^{\alpha} (\sup_{AB} |u|)^{1-\alpha},
\end{equation} 
where $a\in (0,1)$, $A>1$.
Another way to formulate this property is the following:
if $|u|\leq 1$ in $AB$ and $|u|\leq \varepsilon$  in $aB$, then $|u|\leq C \varepsilon^{\alpha}$.
 Using this reformulation and the Harnack chain argument, it is not difficult  to show that iterations of \eqref{eq:3b} imply \eqref{eq:3b*} and iterations of \eqref{eq:3b*} imply \eqref{eq:3b} with worse $C$ and $\alpha$. The proof of this fact can be found in the Appendix (see \ref{a1}).
 \end{remark} 
 
 The  three balls inequality holds for linear (uniformly) elliptic operators 
$$P=\sum_{0\leq|\alpha| \leq 2m } a_\alpha \partial^\alpha$$in $\mathbb{R}^n$, whose coefficients are sufficiently smooth (in particular $C^{2m}$ smoothness of the coefficients is enough , see \cite{Si}). As far as we know, it is an open question to determine the exact regularity conditions on the coefficients of linear elliptic operators of higher order that force the three balls inequality to hold. For second order linear elliptic operators, it is sufficient that the coefficients of highest order are Lipschitz (and elliptic) and the coefficients of lower order are bounded (see \cite{GL1,GL2}). Note that the three balls inequality fails if the coefficients are assumed to be merely Hölder continuous (see \cite{Mi}).

\section{Inequality for two  balls and a wild set.} \label{sec:3b*}
 
We will say that a strong quantitative unique continuation property  (SQUCP)
holds for solutions $u$ of some differential equation $Pu=0$ in a bounded domain $\Omega \subset \mathbb{R}^n$ if the  following is true for some $c,C>0$ and $\alpha \in (0,1)$: For every ball $B$ such that $2B$ is a  subset of $\Omega$ and for every measurable subset $S$ of $B$ with $$|S|> c |B| $$ (we denote the $n$-dimensional Lebesgue measure of $S$ by $|S|$), the inequality

\begin{equation} \label{eq:3b**}
\sup_B |u| \leq C (\sup_{S} |u|)^{\alpha} (\sup_{2B} |u|)^{1-\alpha}
\end{equation} 
holds for any solution to $Pu=0$ in $\Omega$. Sometimes instead of SQUCP we will say "a version of the three balls theorem for wild sets." 

\definecolor{ffffff}{rgb}{1.,1.,1.}
\definecolor{qqzzff}{rgb}{0.,0.6,1.}
\definecolor{ffqqqq}{rgb}{1.,0.,0.}
\begin{center}
\begin{tikzpicture}[line cap=round,line join=round,>=triangle 45,x=0.5cm,y=0.5cm]
\clip(-2.,-2.) rectangle (12.,12.);
\draw [line width=0.4pt,color=ffqqqq] (5.,5.) circle (1.5cm);
\draw [line width=0.4pt] (5.,5.) circle (3.cm);
\draw [rotate around={-61.00740429609983:(6.532021689678707,5.473182052611222)},line width=0.4pt,color=qqzzff,fill=qqzzff,fill opacity=1.0] (6.532021689678707,5.473182052611222) ellipse (0.7679266179030014cm and 0.42435928374177756cm);
\draw [rotate around={31.888107955519008:(4.255528573732505,5.189806288799964)},line width=0.4pt,color=qqzzff,fill=qqzzff,fill opacity=1.0] (4.255528573732505,5.189806288799964) ellipse (0.6309239456180165cm and 0.5546352997206168cm);
\draw [rotate around={-2.8624052261117443:(5.607743647622006,3.3194891003020395)},line width=0.4pt,color=qqzzff,fill=qqzzff,fill opacity=1.0] (5.607743647622006,3.3194891003020395) ellipse (0.45008230116525944cm and 0.2740672469157883cm);
\draw [rotate around={-44.35625428582464:(6.657590406230128,5.4570930132232895)},line width=0.4pt,color=ffffff,fill=ffffff,fill opacity=1.0] (6.657590406230128,5.4570930132232895) ellipse (0.3252967736339419cm and 0.25493622343433453cm);
\draw [color=qqzzff](5.1067115881133205,4.898602449601508) node[anchor=north west] {$S$};
\draw (3.6846781568219416,-0.9672854544754311) node[anchor=north west] {$SQUCP$};
\begin{scriptsize}
\draw[color=ffqqqq] (3.052663298470218,7.890797794610452) node {$B$};
\draw[color=black] (1.4528756882674163,10.517609549634805) node {$2B$};
\end{scriptsize}
\end{tikzpicture}
\end{center}

We note that \eqref{eq:3b**} holds for solutions of second order equations in divergence form in $\mathbb{R}^n$:
$$\textup{div}(A\nabla u )=0,$$
where the matrix $A$  is assumed to be uniformly elliptic
and its coefficients are Lipschitz  in  $\Omega$. The version of the three balls theorem for wild sets is non-trivial when the coefficients are assumed to be just smooth (but not real-analytic), and it is related to a recently solved conjecture of Landis  (\cite{KL88}, p.169) and a conjecture of Donnelly and Fefferman \cite{DF2}, see \cite{LM1,LM2} for the proofs.

We also note that \eqref{eq:3b**} holds for 
for linear elliptic operators 
$$P=\sum_{0\leq|\alpha| \leq 2m } a_\alpha \partial^\alpha$$in $\mathbb{R}^n$, whose coefficients are real-analytic (and elliptic) in a neighborhood of $\Omega$. The last statement is well-known to specialists, but there is no exact reference for this result, and one has to look at several sources to make such a conclusion. A reading guide is provided in the Appendix (see \ref{a2}). 

It is likely that \eqref{eq:3b**} also holds if the coefficients $a_\alpha$ are assumed to be sufficiently smooth.

 \section*{Main results.}
 
 \begin{theorem}
Let $M: (-1, 1) \to (0, \infty)$ be a measurable function
with $$ \int_{-1}^{1} \log^+ \log^+ M(y) dy < \infty$$ and let $P$ be a differential operator. Let $$\Omega \subset B_1(0) \subset \mathbb{R}^n= \{ (x,y): x \in \mathbb{R}^{n-1},  y \in \mathbb{R} \}.$$ 
 Consider the family $L_{M}$ of functions $u$ in $\Omega$  that satisfy the differential equation $Pu = 0$ and the inequality $|u(x,y)| \leq M(y)$ in $\Omega$.

\begin{enumerate}
\item[A)]
If a version of the three balls inequality for wild sets \eqref{eq:3b**} holds for solutions to $Pu=0$ in $\Omega$, 
then $L_{M}$ is uniformly bounded on any compact subset of $\Omega$. 
\item[B)]
If the three balls inequality holds for solutions to $Pu=0$ in $\Omega$ and if additionally $M(y)=M(|y|)$ is a symmetric function and monotonically decreasing on $(0,1)$, then $L_{M}$ is uniformly bounded on any compact subset of $\Omega$. 
\end{enumerate}

\end{theorem}
 
 As previously mentioned, the three balls inequality  holds for linear elliptic differential equations with sufficiently smooth coefficients, and therefore, $\textup{B)}$ can be applied to them. However, one has to assume additionally that $M$ is monotone.
 The monotonicity assumption on $M$ can be removed if the version of the three balls theorem for wild sets  \eqref{eq:3b**} holds for $P$. This property holds for  second order elliptic operators in divergence from with Lipschitz coefficients and linear elliptic operators of arbitrary order with real-analytic coefficients. 
 
 \section{The proof of the main result.}
 Let $K$ be a compact subset of $\Omega$ and let $d = \textup{dist}(\partial \Omega, K)$ denote the distance from $K$ to $\partial \Omega$ . We would like to show 
 that there is a constant $A=A(d,M,P)$ such that 
 any solution $u$ to $Pu=0$ satisfying $|u(x,y)| \leq M(y)$
in $\Omega$ is uniformly bounded: $|u(x,y)| \leq A$ on $K$. 
 
 Assume that $p\in K$ and $|u(p)|> e^{e^{2N}}$ for a large   integer $N$.
 We will construct a sequence of points $p_k$ inductively
 with the property $$|u(p_k)|\geq e^{e^{ak+N}},$$
 where  $a \in (0,1)$ is a small positive constant, which will be defined later.
 The sequence $p_k$ will not start  from $k=1$, but from $k=N$, and we put $p_N=p$.
 
 The construction is inductive. Given $p_k$, we will find $p_{k+1}$.
 Consider the set $E_k = \{x \in \Omega: |u(x)| > e^{e^{ak  }}  \}$. Since $|u(x,y)| \leq M(y)$, the last set is contained in 
 $R_k=\{(x,y) \in B_1(0) : M(y) \geq e^{e^{ak }}  \}$.
Define the distribution function $F$ of $\log^+|M|$:
$$F(t):=\lambda_1(y \in (-1,1): \log^+ M(y)\geq t),$$ where $\lambda_1$ denotes the one-dimensional Lebesgue measure.
The projection of $R_k$ onto the $y$-axis has one-dimensional Lebesgue measure $F(e^{ak})$, and therefore, the projection of $E_k$ onto the $y$-axis has one-dimensional Lebesgue measure not greater than $F(e^{ak})$.

\definecolor{qqttcc}{rgb}{0.,0.2,0.8}
\definecolor{zzttqq}{rgb}{0.6,0.2,0.}
\begin{center}
\begin{tikzpicture}[line cap=round,line join=round,>=triangle 45,x=1.0cm,y=1.0cm]
\clip(-1.7,-1.1) rectangle (2.5,3.3);
\fill[line width=0.pt,color=zzttqq,fill=zzttqq,fill opacity=0.10000000149011612] (-10.,2.) -- (-10.,1.) -- (8.,1.) -- (8.,2.) -- cycle;
\draw [rotate around={63.43494882292217:(0.4880042808072267,1.0909905755854472)},line width=0.4pt] (0.4880042808072267,1.0909905755854472) ellipse (2.0797076269495016cm and 1.7536202022079725cm);
\draw [line width=0.4pt,domain=-1.7:2.5] plot(\x,{(--36.-0.*\x)/18.});
\draw [line width=0.4pt,domain=-1.7:2.5] plot(\x,{(--18.-0.*\x)/18.});
\draw [rotate around={-52.19069729554537:(1.1394787644787645,1.4618478368478367)},line width=1.2pt,dash pattern=on 1pt off 1pt,color=qqttcc,fill=qqttcc,fill opacity=0.5] (1.1394787644787645,1.4618478368478367) ellipse (0.38800843322968276cm and 0.3179975036056134cm);
\draw [line width=0.4pt] (-0.002176317977117037,-1.1) -- (-0.002176317977117037,3.3);
\draw [line width=0.4pt,domain=-1.7:2.5] plot(\x,{(-0.-0.*\x)/3.0400992277724908});
\draw (1.9258773268450656,-0.020106179766621785) node[anchor=north west] {$\overrightarrow{x}$};
\draw (-0.5192647526932077,3.337315176455448) node[anchor=north west] {$y$};
\draw [->,line width=0.4pt] (0.,1.) -- (0.002176317977117037,2.);
\draw [->,line width=0.4pt] (0.002176317977117037,2.) -- (0.,1.);
\begin{scriptsize}
\draw[color=black] (-1.519848058058115,0.40150548823822335) node {$\Omega$};
\draw[color=zzttqq] (-1.4833400026040529,1.7669084130612538) node {$R_k$};
\draw[color=qqttcc] (1.6816255866035668,1.7359644374278707) node {$E_k$};
\draw[color=black] (-0.5181105701835088,1.5237925257568568) node {$F(e^{ak})$};
\end{scriptsize}
\end{tikzpicture}
\end{center}

We will first explain case A) of the theorem.
Consider the ball $B=B_{r_k}(p_k)$, where $r_k= F(e^{ak})$. Assume that 
$2B\subset \Omega$. Now, the set $S_k= \{ (x,y) \in B: |u(x,y)| \leq e^{e^{ak}} \} = B \setminus E_k$ has $n$-dimensional measure $|S_k|> c |B|$. One can take $c=1/4^n$. Indeed, note that $|S_k| = |B| - |B \cap E_k|$, and since the projection of $E_k$ onto the $y$-axis has one-dimensional Lebesgue measure not greater than $F(e^{ak})$, we get that $|B \cap E_k| \leq |D_k|$, where $D_k$ is the intersection of $B$ with a strip of width $F(e^{ak})$ as shown by the figure below. Finally, observe that one can fit $\frac{1}{4}B$ inside $B \setminus D_k$. Therefore, the constant $c = 1/4^n$ does the job.

\begin{center}
\includegraphics[scale = 0.7]{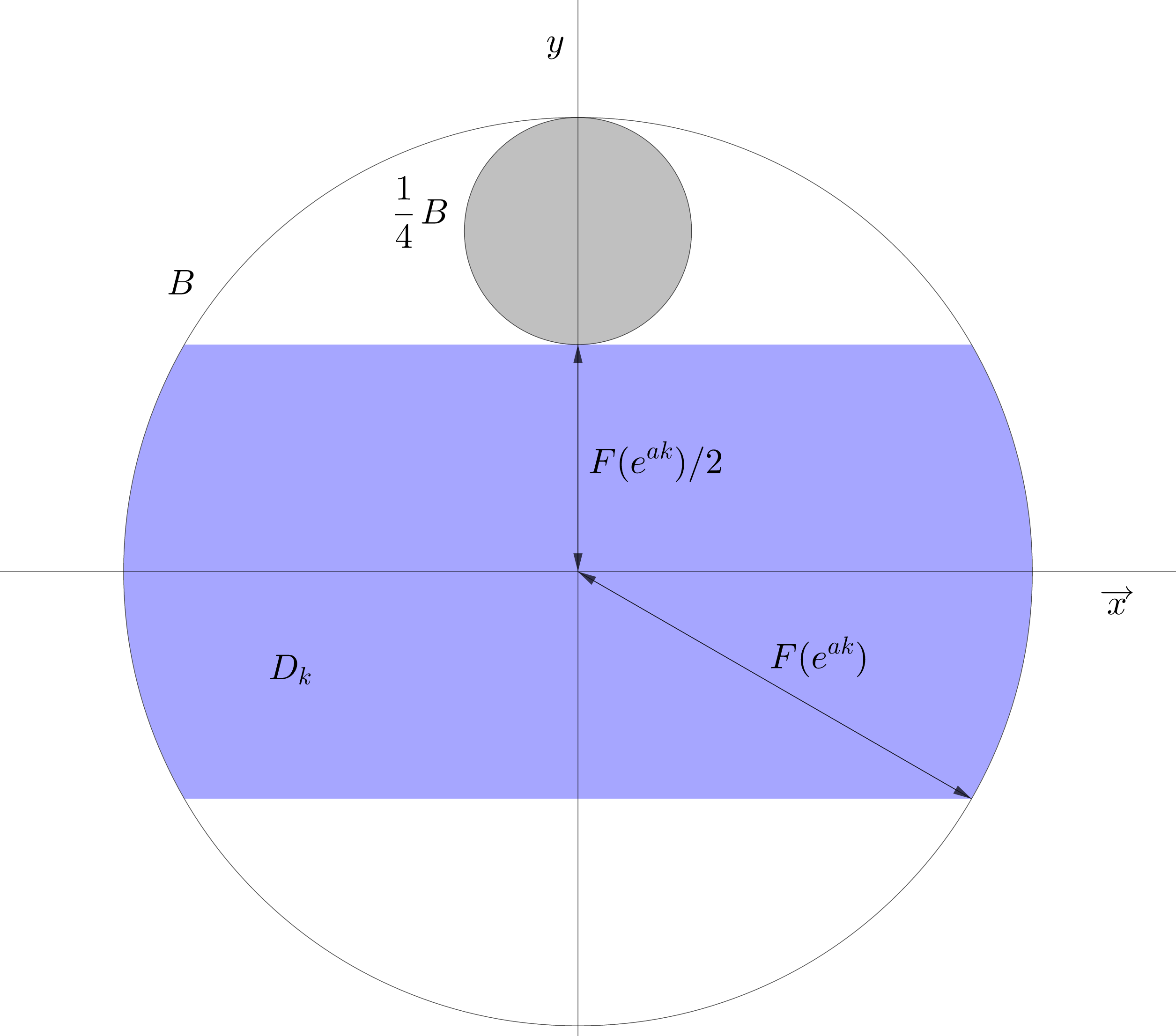}
\end{center}

At the same moment
$$\sup_B |u| \geq |u(p_k)| \geq e^{e^{ak+N}}.$$
Applying the version of the three balls inequality \eqref{eq:3b**} for $S_k$, $B$, $2B$, and $u$, we obtain
\begin{equation} \label{proof:3b}
(1-\alpha) \log \sup_{2B} |u| + \alpha \log \sup_{S_k} |u| \geq \log \sup_B |u| - \log C. 
\end{equation}

Observe that 
\begin{equation} \label{proof:3b2}
 \frac{\alpha}{2} \log \sup_B |u| > \alpha \log \sup_{S_k} |u| + \log C 
 \end{equation}
 because 
 $$  \log \sup_B |u|  \geq e^{ak+N} \geq 2e^{ak} + \frac{2\log C}{\alpha} $$
 for $N$ large enough.
 Combining \eqref{proof:3b} and \eqref{proof:3b2}, we conclude that 
 $$ \log \sup_{2B} |u| > \frac{1-\alpha/2}{1-\alpha} \log \sup_{B} |u| \geq \frac{1-\alpha/2}{1-\alpha} e^{ak+N}.  $$
 Now, we make the choice of $a\in (0,1)$: 
 $$ e^a = \frac{1-\alpha/2}{1-\alpha}.$$
 Note that we can safely assume $\alpha \leq 1/2$ (if $\alpha > 1/2 $, then the three balls inequality also holds with $\alpha = 1/2$ because $\sup_{S_k} |u| \leq \sup_{2B} |u|$).
 This choice of $a$ implies that for any point $p_k \in K$ with $$|u(p_k)| \geq e^{e^{ak+N}}$$ 
 and $$\textup{dist}(p_k,\partial \Omega)> 2F(e^{ak}),$$
 there is a point $p_{k+1}$ with 
 $|u(p_{k+1})| \geq e^{e^{a(k+1)+N}}$ and
 $$|p_{k+1}-p_{k}| < 2F(e^{ak}).$$
 
 We start our sequence of points $p_k$ with $p_N=p$ and 
 the sequence $p_k$ converges to a point in $\Omega$ if 
 $$ \sum_{k=N}^{\infty} 2F(e^{ak})< \textup{dist}(p,\partial \Omega).$$
 
 All that remains is to show that 
 $\sum_{k=0}^{\infty} 2F(e^k)<\infty $. Let $f(y)=\log^+M(y)$.

 \begin{lemma} Let $f(y)$ be a measurable function on $(-1,1)$ and $a>0$. Then $\int_{-1}^1 \log^+ |f|< \infty$ if and only if the following condition holds for the distribution function  $F(y)= \lambda_1(\{y \in (-1,1)\colon  |f(y)|\geq t\})$:
 $$ \sum_{k=0}^{\infty} F(e^{ak})<+\infty.$$ 
\end{lemma} 
 \begin{proof}
 $F(t)$ is a monotonically  decreasing function. Note that
 $$\int_{-1}^1 \log^+ |f| = \sum_{i=0}^{\infty} \int_{E_i} \log^+ |f|,  \textup{ where } E_i=\{y \in (-1,1): e^{ai}\leq|f(y)| < e^{a(i+1)} \}. $$
 The first integral $\int_{E_0} \log^+ |f|$ is always finite, while for $i\geq 1$, 
 $$\int_{E_i} \log^+ |f| \sim i \lambda_1(E_i).$$
  Finally,
 $$  \sum_{i=1}^{\infty} i \lambda_1(E_i) = \sum_{k=1}^{\infty} \sum_{i=k}^{\infty}  \lambda_1(E_i) =\sum_{k=1}^{\infty} F(e^{ak}).$$
 \end{proof}
 
 We come back to case A). Since $\sum_{k=0}^{\infty} 2F(e^k)<\infty $, we see that for sufficiently large $N$
 $$ \sum_{k=N}^{\infty} 2F(e^k)< \textup{dist}(K,\partial \Omega) \leq \textup{dist}(p,\partial \Omega).$$
 The sequence of points $p_k$ converges to a point in $\Omega$, while the values of $|u(p_k)|$ diverge. This contradicts the continuity of $u$. Therefore, $|u| < e^{e^{2N}}$ on $K$. Thus, $u$ admits a uniform bound on every compact subset $K$ of $\Omega$, which depends on $\textup{dist}(K,\partial \Omega)$, on $M$, and on the constants in the three balls inequality for wild sets, determined by the differential operator $P$.
 
 Case B) is similar to case A). We assume the contrary and construct a Cauchy sequence of points with divergent values. There is only one change in the proof. Namely, we are allowed to apply only the standard version \eqref{eq:3b*} of the three balls theorem.  Given $p_k$, we need to find a ball nearby where the values of $|u|$ are smaller than $e^{e^{ak}}$. Surely, we cannot take a ball centered at $p_k$ because $|u(p_k)| \geq e^{e^{ak+N}}$. Let $p_k=(x_k,y_k)$, where $x_k \in \mathbb{R}^{n-1}$ and $y_k \in \mathbb{R}$. The monotonicity of the majorant $M$ allows us to conclude that
 the set $R_k=\{(x,y): M(y) \geq e^{e^{ak }}  \}$ is 
 contained in $\{ (x,y): |y| \leq  F(e^{ak})/2 \}$ and therefore $$|y_k|\leq F(e^{ak})/2.$$
Assume that $\textup{dist}(p_k,\partial \Omega) > 20F(e^{ak})$. Then we can take a ball $B$ centered at $q_k=(x_k,2F(e^{ak}))$ and of radius $4F(e^{ak})$. The ball $B$ contains $p_k$. So $$\sup_B |u| \geq e^{e^{ak+N}},$$ while the four times smaller ball $\frac{1}{4}B= B_{F(e^{ak})}(q_k)$ does not intersect $R_k$, and
therefore,
$$ \sup_{\frac{1}{4}B}|u| \leq e^{e^{ak}}.$$
Applying three balls inequality for $\frac{1}{4}B$,$B$,$2B$,
and $u$ and proceeding in a similar fashion to case $A)$,
we conclude that there is a point $p_{k+1} \in 2B$
with 
 $|u(p_{k+1})| \geq e^{a(k+1)+N}$ and
 $$|p_{k+1}-p_{k}| < 20F(e^{ak}).$$
 The rest of the argument is similar.
 The sum $\sum_{k=N}^\infty 20F(e^{ak})< \textup{dist}(p,\partial \Omega)$ for $N$ large enough. The sequence of points $p_k$ converges to a point in $\Omega$, while the values of $|u(p_k)|$ diverge, which implies a contradiction.
 \QEDA
 
  \section{Appendix}
\subsection{Three Balls Inequality for Arbitrary Ratios} \label{a1}

Let $\Omega \subseteq \R^n$ be a domain. Given a function $u \colon \Omega \to \R$ and $B_{r}(x) \subseteq \Omega$, let $M_x(r) \defeq \sup_{{B}_{r}(x)} \abs{u}$.
\begin{lemma}
With the notation above, assume that for some ratios $1 < a < b$ and some constants $C> 0$ and $\alpha \in (0,1)$ the three balls inequality
\begin{equation*} \label{ineq}
    M_x(ar) \leq C M_x(r)^{\alpha} M_x(br)^{1-\alpha}
\end{equation*}
holds whenever $B_x(br) \subseteq \Omega$. Then, the three balls inequality will continue to hold for any other set of ratios of radii (with different $C$ and $\alpha$ that depend on the ratios).
\end{lemma}

\begin{proof}
Let $1 < a' < b'$ be a different set of ratios. Consider two balls $B_r(x) \subset B_R(x) \subset \Omega$ such that $R/r = b'$. In what follows, we present a step by step expansion of the smallest ball $B_r(x)$ until it reaches the radius $\rho = a'r$.

\textbf{\textit{Step 1.}} Consider a ball $B_{\rho_1}(x_1)$ inside $B_r(x)$. We would like to apply the three balls inequality for $B_{\rho_1}(x_1) \subset B_{a\rho_1}(x_1) \subset B_{b\rho_1}(x_1)$. For that reason we want $B_{b\rho_1}(x_1) \subset B_R(x)$, so this gives us the constraint $\ds \rho_1 \leq \frac{R-r}{b-1}$. Now, $M_{x_1}(\rho_1) \leq M_{x}(r)$ and $M_{x_1}(b\rho_1) \leq M_{x}(R)$. Therefore, applying the three balls inequality for $B_{\rho_1}(x_1) \subset B_{a\rho_1}(x_1) \subset B_{b\rho_1}(x_1)$ we have
    \begin{equation*}
        M_{x_1}(a\rho_1) \leq C M_x(r)^{\alpha} M_x(R)^{1-\alpha}.
    \end{equation*}
We can safely assume $C > 1$, and if we allow $x_1$ to vary along $\partial B_{r-\rho_1}(x)$ (assuming $\rho_1 < r$), we get that
    \begin{equation*}
        M_x(r+(a-1)\rho_1) \leq C M_x(r)^{\alpha} M_x(R)^{1-\alpha}.
    \end{equation*}
    
\definecolor{ffffff}{rgb}{1.,1.,1.}
\definecolor{uuuuuu}{rgb}{0.26666666666666666,0.26666666666666666,0.26666666666666666}
\definecolor{ffwwzz}{rgb}{1.,0.4,0.6}
\definecolor{xdxdff}{rgb}{0.49019607843137253,0.49019607843137253,1.}
\definecolor{qqttzz}{rgb}{0.,0.2,0.6}
\definecolor{ududff}{rgb}{0.30196078431372547,0.30196078431372547,1.}
\begin{center}
\begin{tikzpicture}[line cap=round,line join=round,>=triangle 45,x=0.48cm,y=0.48cm]
\clip(-0.2,-2.) rectangle (24.7,11.);
\draw [line width=0.4pt,color=ffwwzz,fill=ffwwzz,fill opacity=0.05000000074505806] (19.441961263685982,5.) circle (2.4269983594690028cm);
\draw [line width=2.pt,color=ffffff,fill=ffffff,fill opacity=1.0] (16.,5.) circle (2.4cm);
\draw [line width=0.4pt,color=qqttzz] (5.,5.) circle (2.4cm);
\draw [line width=0.4pt,color=qqttzz] (5.,5.) circle (0.96cm);
\draw [line width=0.4pt,color=ffwwzz] (6.240347337739723,5.) circle (0.36463327788493327cm);
\draw [line width=0.4pt,color=ffwwzz] (6.240347337739723,5.) circle (0.7578773140191033cm);
\draw [line width=0.4pt,color=ffwwzz] (6.240347337739723,5.) circle (1.3246332778849328cm);
\draw [line width=0.4pt,color=qqttzz] (16.,5.) circle (2.4cm);
\draw [line width=0.4pt,color=ffwwzz] (19.441961263685982,5.) circle (0.7478585934307284cm);
\draw [line width=0.4pt] (16.,5.)-- (11.,5.);
\draw [line width=0.4pt] (19.441961263685982,5.)-- (21.,5.);
\draw [line width=0.4pt] (21.,5.)-- (24.498207845913072,5.);
\draw [line width=0.4pt,color=ffwwzz] (19.441961263685982,5.) circle (0.7478585934307284cm);
\draw (3.498342150889992,-0.274119595765216) node[anchor=north west] {$Step$ $1$.};
\draw (15.992738788988431,-0.274119595765216) node[anchor=north west] {$Step$ $n+1$.};
\begin{scriptsize}
\draw [fill=ududff] (5.,5.) circle (1.0pt);
\draw[color=ududff] (5.096331425739354,5.477850837892327) node {$x$};
\draw[color=qqttzz] (2.12725639780368,9.936362843538292) node {$B_R(x)$};
\draw[color=qqttzz] (3.2423234059357707,7.155478162967053) node {$B_r(x)$};
\draw [fill=xdxdff] (6.240347337739723,5.) circle (1.0pt);
\draw[color=xdxdff] (6.321197361356383,5.46805191040739) node {$x_1$};
\draw [fill=ududff] (16.,5.) circle (1.0pt);
\draw[color=ududff] (15.953543079048687,5.517046547832072) node {$x$};
\draw[color=qqttzz] (12.827685211354032,9.799177858749186) node {$B_n$};
\draw [fill=xdxdff] (19.441961263685982,5.) circle (1.0pt);
\draw[color=xdxdff] (19.294977351411937,5.566041185256752) node {$x_{n+1}$};
\draw [fill=xdxdff] (24.498207845913072,5.) circle (1.0pt);
\draw [fill=uuuuuu] (11.,5.) circle (1.0pt);
\draw[color=black] (14.721090013612426,6.212770399262541) node {$r + (a-1)(\rho_1 + ... + \rho_n)$};
\draw [fill=uuuuuu] (21.,5.) circle (1.0pt);
\draw[color=black] (20.000500130327346,4.429365597004154) node {$\rho_{n+1}$};
\draw[color=black] (22.70379768107936,5.703226170045859) node {$(a-1)\rho_{n+1}$};
\end{scriptsize}
\end{tikzpicture}  
\end{center}

\textbf{\textit{Step from $n$ to $n+1$.}} We have expanded the smallest ball to $$B_n = B_{r + (a-1)(\rho_1 + ... + \rho_n)}(x)$$ where the constraints $\rho_n \leq r/2$
ensure that $B_{\rho_n}(x_n)\subset B_n$ and the constraints $$\ds \rho_{i+1} \leq \frac{R-r-(a-1)(\rho_1 + \dots + \rho_i)}{b-1}$$ ensure that all of the expansion balls $B_{b\rho_n}(x_n)$ are inside $B_R(x)$ at each step. We have also deduced the inequality
    \begin{equation*}
        M_x(r+(a-1)(\rho_1 + \dots + \rho_n)) \leq C^{\frac{1-\alpha^n}{1-\alpha}} M_x(r)^{\alpha^n} M_x(R)^{1-\alpha^n}.
    \end{equation*}

 Choose $x_{n+1}$ at a distance $\rho_{n+1}$ from the boundary of $B_n$ where $$\ds \rho_{n+1} \leq \frac{R-r-(a-1)(\rho_1 + \dots + \rho_n)}{b-1}.$$ Thus $$M_{x_{n+1}}(\rho_{n+1}) \leq M_x(r+(a-1)(\rho_1 + \dots + \rho_n))$$ and $M_{x_{n+1}}(b\rho_{n+1}) \leq M_x(R)$. We apply the three balls inequality:
\begin{gather*}
    M_{x_{n+1}}(a\rho_{n+1}) \leq C M_{x_{n+1}}(\rho_{n+1})^{\alpha}M_{x_{n+1}}(b\rho_{n+1})^{1-\alpha}, \\
     M_{x_{n+1}}(a\rho_{n+1}) \leq C [C^{\frac{1-\alpha^n}{1-\alpha}} M_x(r)^{\alpha^n} M_x(R)^{1-\alpha^n}]^\alpha M_{x_{n+1}}(b\rho_{n+1})^{1-\alpha}, \\
     M_{x_{n+1}}(a\rho_{n+1}) \leq C^{\frac{1-\alpha^{n+1}}{1-\alpha}} M_x(r)^{\alpha^{n+1}} M_x(R)^{1-\alpha^{n+1}}.
\end{gather*}
Now, allowing $x_{n+1}$ to range over $\partial B_{r + (a-1)(\rho_1 + ... + \rho_n) - \rho_{n+1}}(x)$, we obtain
\begin{equation*}
    M_x(r+(a-1)(\rho_1 + \dots + \rho_{n+1})) \leq 
    C^{\frac{1-\alpha^{n+1}}{1-\alpha}} M_x(r)^{\alpha^{n+1}} M_x(R)^{1-\alpha^{n+1}},
\end{equation*}
which completes the inductive construction.

What remains to be shown is that there exists $n$ such that $$r+(a-1)(\rho_1 + \dots + \rho_n) \geq a'r.$$ If that is the case, then $M_x(a'r) \leq M_x(r+(a-1)(\rho_1 + \dots + \rho_n))$ and therefore $$M_x(a'r) \leq C^{\frac{1-\alpha^n}{1-\alpha}} M_x(r)^{\alpha^n} M_x(b'r)^{1-\alpha^n},$$ and we are done.

So let $$\ds \rho_{i+1} = \min(r/2,\frac{R-r-(a-1)(\rho_1 + \dots + \rho_i)}{b-1}), \quad \forall i \in \N.$$ 

For every $\rho$ such that $r < \rho < R$, there exists $n$ for which $$r+(a-1)(\rho_1 + \dots + \rho_n) \geq \rho.$$
The last statement can be shown by assuming the contrary.
Indeed, if $$ R-r-(a-1)(\rho_1 + \dots + \rho_n)\geq R - \rho,$$
then 
$\rho_{n}$ are uniformly bounded below by the positive constant $\min(r/2, \frac{R-\rho}{b-1}) $ and therefore
$r+(a-1)(\rho_1 + \dots + \rho_n)$ will be bigger than $ \rho$ for sufficiently large $n$.

\end{proof}

\subsection{Real-Analyticity and Three Balls Inequality} \label{a2}
   Let $L=\sum_{|\beta|\leq m} a_{\beta} \partial^\beta$ be a linear elliptic differential operator with real-analytic coefficients in a unit ball $B=B_1(0)\subset \mathbb{R}^n$ such that
   \begin{equation} \label{eq: analyticity}
   \sum_{|\beta|\leq m} \sup_{B} |\partial^\alpha a_\beta| \leq A^{|\alpha|} \alpha!
   \end{equation}
   for some constant $A$.
   It is well-known (but not easy to find a good reference to) that any solution to $Lu=0$ is real analytic. In \cite{LH} (p.178-180), there is a quantitative and uniform version of this statement, which implies that there is $C>0$ and $r>0$ such that
   \begin{equation} \label{eq: Hormander}
    \sup_{B_r} |\partial^{\alpha} u| \leq \sup_{B} |u| C^{|\alpha|+1} \alpha!.
    \end{equation}
    It allows us to conclude that $u$ has a holomorphic extension to some complex neighborhood of zero:
    there is $r>0$ such that $u$ can be extended to the complex ball
    $B_{\rho}^{\mathbb{C}} \subset \mathbb{C}^n$ with center at the origin of radius some $\rho$, which depends on the radius of convergence of the Taylor series. The holomorhpic extension comes with the estimate:
    \begin{equation} \label{eq: holomorphic extension}
    \sup \limits_{B_{\rho}^{\mathbb{C}}}|u| \leq C \sup \limits_{B}|u|.
    \end{equation}

    \begin{remark} Donelly and Fefferman (see \cite{DF}) applied the Cauchy-type estimate \eqref{eq: Hormander}  to Laplace eigenfunctions
    $\varphi_\lambda:$ $\Delta \varphi_\lambda + \lambda \varphi_\lambda=0$
    on Riemannian manifolds with real-analytic metrics
    in the setting when $\lambda \to \infty$. One may object that the coefficients of the equation tend to infinity; however, if one considers a geodesic ball of radius $1/\sqrt \lambda$, the $\sqrt \lambda $ rescaled equation has uniform bounds for the coefficents and their derivatives.
  It implies that $\varphi_\lambda$ has a holomorphic extension to a complex neighborhood  of radius 
  $c/\sqrt \lambda$ of any point $x$ on the manifold with estimate in local coordinates:
  $$ \sup\limits_{B_{c/\sqrt \lambda}^\mathbb{C}(x)} | \varphi_\lambda | \leq C \sup\limits_{B_{1/\sqrt \lambda}(x)} | \varphi_\lambda |.$$
   \end{remark}
   As in the work of Donnelly and Fefferman, we also need a scaled version of \eqref{eq: analyticity} and \eqref{eq: holomorphic extension}.
   Given a ball $B_\rho(x_0)\subset B_1(0)$, any solution to $Lu=0$ in $B_\rho(x_0)$ satisfies the Cauchy-type estimate:
   \begin{equation}
   |\partial^\alpha u(x_0)| \leq \sup\limits_{B_\rho(x_0)} |u| C^{|\alpha|+1} \alpha! / \rho^{|\alpha|}. 
   \end{equation}
   and the Taylor series of $u$ converge in a complex neighborhood of $x_0$ of radius $c\rho$:
   \begin{equation} \label{eq: local holomorphic extension}
   \sup\limits_{B_{c\rho}^\mathbb{C}(x_0)} | u | \leq  C \sup\limits_{B_{\rho}(x_0)} | u |,
   \end{equation}
   where $c>0$ is a small numerical constant and $C>0$ is a large numerical constant.
   
   The holomorhic extension with estimate for solutions to $Lu=0$ implies the three balls theorem \eqref{eq:3b} and its version for wild sets \eqref{eq:3b**} (see \cite{SV}, Theorem 1).

\end{document}